\documentclass[11pt]{amsart} 
\usepackage{amssymb,amsmath,latexsym,enumerate,graphicx}

\newtheorem{theorem}{Theorem} 
\newtheorem{corollary}[theorem]{Corollary}

\newtheorem{lemma}[theorem]{Lemma}

\newtheorem*{rem}{Remarks}

\newtheorem*{exam}{Examples}

\def\Z{\mathbb{Z}}

\def\C{\mathbb{C}}

\title{Counting Group Valued Graph Colorings}
\date{4 April, 2012}
\begin{document}
\begin{abstract}
There are many variations on partition functions for graph homomorphisms or colorings.  The case considered here is a counting or hard constraint problem in which the range or color graph carries a free and vertex transitive Abelian group action so that the colors are identified with the elements of this group.  
A Fourier transform is used to obtain an expansion for the numbers of colorings with terms indexed by isthmus free subgraphs of the domain.  The terms are products of a polynomial in the edge density $\alpha$ of the color graph and the number of colorings of the indexing subgraph of the domain into the complementary color graph.  The polynomial in $\alpha$ is independent of the color group and the term has order $(1-\alpha)^r$ where $r$ is the number of vertices minus the number of components in the indexing subgraph.  
Thus if $1-\alpha $ is small there is a main term indexed by the empty subgraph which is a polynomial in $\alpha $ and the first dependence on the coloring group occurs in the lowest order corrections which are indexed by the shortest cycles in the graph and are of order $(1-\alpha)^{g-1}$ where $g$ is the length of these shortest cycles.  The main theorem is stated as a reciprocity law.  Examples are given in which the coloring groups are long cycles and products of short cycles and adjacent vertices are required to have distant rather than distinct colors.
The chromatic polynomial of a graph corresponds to using any group and taking the allowed set to be the complement of the identity.  
\end{abstract}

\maketitle

\section{Introduction}
There are many variations on partition functions for graph homomorphisms or colorings [1, 2].  The case considered here is a counting or hard constraint problem in which the range or color graph carries a free and vertex transitive Abelian group action so that the colors are identified with the elements of this group.  

Fix a finite set $V$, a finite Abelian group ${F}$ and a subset  ${A}=-{A}\subseteq {F}$. 
The set $V$ will be the vertex set for the graphs; the elements of the group $F$ will be the colors applied to these vertices; and the subset $A$ will be the allowed differences between the two colors used for the ends of an edge.  Since $A$ is symmetric edge orientations can be ignored.  
Write $f=|{F}|$, $v=|V|$, ${\alpha}={|{A}|\over f}$, $\overline{A}={F}-{A}$ and $\overline{\alpha} =1-{\alpha}$.  
Write $P_V$ for the partially ordered set of isthmus-free simple graphs with vertex set $V$.  
Thus $$P_V=(\{E\subseteq {V\choose 2}|c(E)=c(E-\{t\})\hbox{ for every }t\in E\},\subseteq)$$ 
where $c(E)$ is the number of connected components of the graph with edge set $E$ and vertex set $V$.  
Write $\delta=\delta_E:{F}^V\rightarrow {F}^E$ for the coboundary map for this graph.  
If $P$ is a finite set, write $\C^P$ for the $\C$-vector space with basis indexed by $P$ and with coordinates $[\cdot]_p:\C^P\rightarrow \C$ and $GL(\C^P)$ for the general linear group with matrix entries $[\cdot]^p_q:GL(\C^P)\rightarrow \C$. 
The focus of this paper is on the probability that a uniformly chosen coloring of the vertices of a graph by elements of the group ${F}$ has all differences along edges in the set ${A}$.  This is summarized in the vector $\Gamma^{{A}}\in {\C}^{P_V}$ with coordinates $$[\Gamma^{{A}}]_E=f^{-v}|\delta^{-1}{A}^{E}|.$$
This can also be viewed in terms of edge colorings since also $$[\Gamma^{A}]_E=f^{c(E)-v}|{A}^{E}\cap \hbox{Im}(\delta)|.$$ 

This vector will be expanded using the linear operators $j$ and $r^e$ in $GL({\C}^{P_V})$. The former is associated to the partial order and has entries $[j]^E_H=1$ if $E\subseteq H$ and $[j]^E_H=0$ otherwise.  The latter is diagonal, associated to the linear extension of $P_V$ given by counting edges and has entries $[r^e]^E_E=r^{|E|}$, where $r$ is any complex number. 
Write $J_r=r^ej(r^{-1})^e$ and $M_r=J_{1-r}(-1)^eJ_r^{-1}$.  These are matrices of polynomials in $r$.  

The main point is a reciprocity formula which can then be phrased as a formula for the probability of an allowed coloring in terms of the probability of totally disallowed ones for subgraphs or as a polynomial approximation independent of the coloring group.    
\vskip7pt
\noindent{\bf Theorem.}
$J_{{\alpha}}^{-1}\Gamma^{{A}} = (-1)^eJ_{\overline{\alpha} }^{-1}\Gamma^{\overline{A}}$.  
\vskip7pt
\begin{corollary}
$\Gamma^{{A}}=M_{\overline{\alpha} }\Gamma^{\overline{A}}$. 
\end{corollary} 
\vskip5pt\noindent Write $g(E)$ for the girth or length of the shortest cycle of $E$.
\begin{corollary}
$[\Gamma^{{A}}]_E= [M_{\overline{\alpha} }]^\emptyset_E+O_{\overline{\alpha} \rightarrow 0}(\overline{\alpha} ^{g(E)-1})$.  
\end{corollary} 
\noindent Write $[\chi(f)]_E=f^v[\Gamma^{{{F}}-\{0\}}]_E$ for the chromatic polynomial of the graph $E$ and $[f^c]_E=f^{c(E)}$.  
\begin{corollary}
$\chi(f)=M_{f^{-1}}f^c$.
\end{corollary} 
\section{Proofs}

Write $\langle .,.\rangle :\hat{F}\times {F}\rightarrow S^1\subseteq {\C}$ for the canonical pairing between $F$ and its (isomorphic) Pontriajin dual and extend this to $\langle .,.\rangle_E:{\hat{F}}^E\times {F}^E\rightarrow S^1\subseteq{\C}$ via $\langle P,Q\rangle_E=\prod_{t\in E}\langle P(t),Q(t)\rangle$ and similarly for $\langle .,.\rangle_V$.

The coboundary map $\delta_E$ has an $\langle .,.\rangle_E $-adjoint 
boundary map $\partial=\partial_E:{\hat{F}}^E\rightarrow {\hat{F}}^V$, so that 
if $P\in {\hat{F}}^E$, $Q\in {F}^E$ and $X\in {F}^V$ then
$\langle P,Q+\delta X\rangle _E=\langle P,Q\rangle _E\langle P,\delta X\rangle _E=\langle P,Q\rangle _E\langle\partial P,X\rangle _V$ and if ${\bf 0}=0^V$ is the $0$ vector in $\hat{F}^V$ then
$$\sum_{X\in {F}^V}\langle Y,X\rangle _V={\Bigg\{}\begin{array}{ll} f^v & \hbox{ if }Y={\bf 0}\\ 0 &\hbox{ otherwise.}\end{array}$$  

Combining these observations and using a Fourier transform gets from a sum 
over vertex colorings to a double sum over edge colorings.

\begin{lemma} $$[\Gamma^{{A}}]_E=f^{-|E|}\sum_{P\in\partial^{-1}({\bf 0})}\hskip10pt\sum_{Q\in {A}^{E}}\langle P,Q\rangle _{E}.$$
\end{lemma}
\begin{proof} Consider the Fourier expansion of a delta function: $$d_{{A}}(x)={\Bigg{\{}}\begin{array}{ll} 1 &\hbox{ if }x\in {A}\\0 &\hbox{ if }x\in {\overline{A}}\end{array}{\Bigg{\}}}  
=f^{-1}\sum_{p\in {\hat{F}}}\hskip3pt\sum_{q\in {A}}\langle p,q-x\rangle.$$ 

Compute:
$$[\Gamma^{{A}}]_E=f^{-v}\sum_{X\in {F}^V}\prod_{\{u,w\}\in E}d_{{A}}([X]_u-[X]_w)$$
                    $$ =f^{-v}\sum_{X\in {F}^V}\prod_{t\in E}d_{{A}}([\delta X]_t)$$
                    $$ =f^{-v}\sum_{X\in {F}^V}\prod_{t\in E}f^{-1}\sum_{p\in {\hat{F}}}\sum_{ q\in {A}}\langle p, q-[\delta X]_t\rangle$$
                    $$ =f^{-|E|-v}\sum_{X\in {F}^V}\sum_{ P\in {\hat{F}}^E}\sum_{ Q\in {A}^E}\langle P,Q-\delta X\rangle_E$$
                    $$ =f^{-|E|-v}\sum_{X\in {F}^V}\sum_{ P\in {\hat{F}}^E}\sum_{ Q\in {A}^E}\langle P,Q\rangle_E\langle P,-\delta X\rangle_E$$ 
                    $$ =f^{-|E|-v}\sum_{X\in {F}^V}\sum_{ P\in {\hat{F}}^E}\sum_{ Q\in {A}^E}\langle P,Q\rangle_E\langle \partial P,-X\rangle_V$$ 
                    $$ =f^{-|E|}\sum_{P\in \partial^{-1}({\bf 0})}\hskip10pt\sum_{ Q\in {A}^E}\langle P,Q\rangle_E.$$ 
\end{proof}
 
Write: 
$$[\Gamma^{{A}}_+]_E=f^{-|E|}\sum_{P\in\partial^{-1} ({\bf 0})\cap ({\hat{F}}-\{0\})^{E}}\hskip10pt\sum_{ Q\in {A}^{E}}\langle P,Q\rangle _{E}.$$  

\begin{lemma} $\Gamma^{{A}}=J_{{\alpha}}\Gamma^{{A}}_+$.  \end{lemma}

\begin{proof} Write $E'$ for the support of $P$ and restrict $Q$ to $E'$ thereby losing a factor of $({\alpha}f)^{|E-E'|}$.  
$$f^{|E|}[\Gamma^{{A}}]_E=\sum_{P\in \partial_E^{-1}({\bf 0})}\hskip10pt\sum_{ Q\in {A}^E}\langle P,Q\rangle_E$$ 
       $$=f^{|E|}{\alpha}^{|E|}\sum_{E'\leq E} {\alpha}^{-|E'|}f^{-|E'|} \sum_{ P'\in \partial_{E'}^{-1}({\bf 0})\cap ({\hat{F}}-\{0\})^{E'}}\hskip10pt\sum_{ Q'\in {A}^{E'}}\langle P',Q'\rangle_{E'}$$ 
       $$ =f^{|E|}[{\alpha}^ej({\alpha}^{-1})^e\Gamma_+^{{A}}]_E.$$ 
\end{proof}

\begin{lemma} $\Gamma^{{A}}_+=(-1)^e\Gamma^{\overline{A}}_+$. \end{lemma}  

\begin{proof} If $p\in {\hat{F}}-\{0\}$ then $\sum_{q\in {F}}\langle p,q\rangle=0$  
so $\sum_{q\in {A}}\langle p,q\rangle=-\sum_{q\in {\overline{A}}}\langle p,q\rangle$.  
Applying this observation $e$ times gives 
$$\sum_{Q\in {A}^E}\langle P,Q\rangle_E=(-1)^e\sum_{Q\in {\overline{A}^E}}\langle P,Q\rangle_E$$ for any nowhere vanishing $P\in ({\hat{F}}-\{0\})^E$.  
The desired equation is the sum of this one over $P$'s with trivial boundary.  
\end{proof}

The theorem now follows immediately and corollaries one and three are immediate consequences.  
For corollary two note that $[\Gamma^{\overline{A}}]_E\leq\overline{\alpha} ^{v-c(E)}$ with equality if $E$ is a forest.  Thus $[\Gamma^{\overline{A}}]_\emptyset=1$ gives the largest term if $\overline{\alpha} $ is small with the first corrections associated to the shortest cycles in $E$.

\section{Examples}

\noindent{\bf Example 1.} ($M$ for $v=3$ and $v=4$)

It is straightforward to compute $M_{\overline{\alpha} }$ as a matrix of polynomials in $\overline{\alpha} $.
If $v=3$ then $P_v$ has two elements (corresponding to the complete and empty graphs with three vertices)  
and{ $$M_{\overline{\alpha} }=
\Big{[}\begin{array}{cc} 1 & (1-\overline{\alpha} )^3\\ 0 & 1\end{array}\Big{]}  
\Big{[}\begin{array}{cc} -1 & 0\\ 0 & 1\end{array}\Big{]} 
\Big{[}\begin{array}{cc} 1 & -\overline{\alpha} ^3\\ 0 & 1\end{array}\Big{]}  
=\Big{[}\begin{array}{cc} -1 & 1-3\overline{\alpha} +3\overline{\alpha} ^2\\ 0 & 1\end{array}\Big{]}.$$}  

If $v=4$ then $P_V$ has fifteen elements, falling into the five graph isomorphism classes: complete, complement of an edge, four cycle, three cycle and empty.  
Use the notation ${{\bf 1}}^a_b$ for the $a$ by $b$ matrix of $1$s, $K^3_6=\left[\begin{array}{c}I_3\\I_3\end{array}\right]$ and $L^4_6=\left[{\tiny\begin{array}{cccc} 1 & 1 & 0 & 0\\ 1 & 0 & 1 & 0 \\ 1 & 0 & 0 & 1 \\ 0 & 0 & 1 & 1 \\ 0 & 1 & 0 & 1 \\ 0 & 1 & 1 & 0\end{array}}\right]$ 
to express $M_{\overline{\alpha}}=J_{{\alpha}}(-1)^eJ_{\overline{\alpha} }^{-1}$ in block form: 
$$\left[\begin{array}{ccccc}
1 & {\alpha}{{\bf 1}}^6_1 & {\alpha}^2{{\bf 1}}^3_1 & {\alpha}^3{{\bf 1}}^4_1 & {\alpha}^6 \\
0 & I_6 & {\alpha}K^3_6 & {\alpha}^2L^4_6 & {\alpha}^5{{\bf 1}}^1_6 \\
0 & 0    & I_3  & 0 & {\alpha}^3{{\bf 1}}^1_3 \\
0 & 0    & 0    & \pm I_4 & c_9{{\bf 1}}^1_4 \\
0 & 0    & 0    & 0    & 1
\end{array}\right]\left[\begin{array}{ccccc}
1 & 0 & 0 & 0 & 0 \\
0 & -I_6 & 0 & 0 & 0 \\
0 & 0    & I_3  & 0 & 0 \\
0 & 0    & 0    & -I_4 & 0 \\
0 & 0    & 0    & 0    & 1
\end{array}\right]\left[\begin{array}{ccccc}
1 & -\overline{\alpha} {{\bf 1}}^6_1 & \overline{\alpha} ^2{{\bf 1}}^3_1 & 2\overline{\alpha} ^3{{\bf 1}}^4_1 & -6\overline{\alpha} ^6 \\
0 & I_6 & -\overline{\alpha} K^3_6 & -\overline{\alpha} ^2L^4_6 & 2\overline{\alpha} ^5{{\bf 1}}^1_6 \\
0 & 0    & I_3  & 0 & -\overline{\alpha} ^4{{\bf 1}}^1_3 \\
0 & 0    & 0    & I_4 & -\overline{\alpha} ^3{{\bf 1}}^1_4 \\
0 & 0    & 0    & 0    & 1
\end{array}\right]$$
$$=\left[\begin{array}{ccccc}
1 & -{{\bf 1}}^6_1 & {{\bf 1}}^3_1 & (-1+3\overline{\alpha}  -\overline{\alpha} ^2+\overline{\alpha} ^3){{\bf 1}}^4_1 & (1-6\overline{\alpha} +15\overline{\alpha} ^2-16\overline{\alpha} ^3) \\
0 & -I_6 & K^3_6 & (-1+2\overline{\alpha} )L^4_6 & (1-5\overline{\alpha} +10\overline{\alpha} ^2-3\overline{\alpha} ^3){{\bf 1}}^1_6 \\
0 & 0    & I_3  & 0 & (1-4\overline{\alpha} +6\overline{\alpha} ^2-4\overline{\alpha} ^3){{\bf 1}}^1_3 \\
0 & 0    & 0    & -I_4 & (1-3\overline{\alpha} +3\overline{\alpha} ^2){{\bf 1}}^1_4 \\
0 & 0    & 0    & 0    & 1
\end{array}\right].$$
\vskip10pt
Next is a brief discussion of two infinite families of color groups.  
The first is a single cycle and the second a product of short (length two) cycles.  In both cases the allowed sets are all elements far enough from the 
identity, with the latter using the Hamming metric.  
Thus $[\Gamma^{{A}}]_E$ is the probability that a coloring has adjacent 
vertices colored with distant rather than distinct colors.  
\vskip15pt
\noindent{\bf Example 2.} (Cyclic group)

Consider ${F}={\Z}\slash f{\Z}$ and ${A}=(k, f-k)$ so $k={\overline{\alpha} f-1\over 2}$.  If $E$ is fixed then $[\Gamma^{{A}}]_E$ is a piecewise polynomial in $\overline{\alpha} $ and $f^{-1}$ of degree at most $e$ in the first and $v$ in the second. 

If $v=3$ then $[\Gamma^{{A}}]_\emptyset=[\Gamma^{\overline{A}}]_\emptyset=1$ and $$[\Gamma^{\overline{A}}]_{K_3}=\Bigg{\{}\begin{array}{cl}
1-3\overline{\alpha} +3\overline{\alpha} ^2 & \hbox{ if }\overline{\alpha} >{2\over 3} \\ {3\over 4}\overline{\alpha} ^2 + {1\over 4}f^{-2} & \hbox{ otherwise}\end{array}\Bigg{\}}$$
so $$[\Gamma^{{A}}]_{K_3}=\Bigg{\{}\begin{array}{cl}
0 & \hbox{ if }\overline{\alpha} >{2\over 3} \\ 1-3\overline{\alpha} +{9\over 4}\overline{\alpha} ^2 - {1\over 4}f^{-2} & \hbox{ otherwise}\end{array}\Bigg{\}}.$$
\vskip15pt
\noindent{\bf Example 3.}  (Hamming)

Consider ${F}=({\Z}\slash 2{\Z})^n$ and ${A}=\{x\in {F}||\{i|[x]_i=1\}| >k\}$ so $f=2^n$ and if $k={n\over 2}+{r\over 2}\sqrt{n}$ then
$\lim_{n\rightarrow\infty}\overline{\alpha} ={1\over\sqrt{2\pi}}\int_{t=-\infty}^{-r}e^{-t^2\over 2}dt$.  
The central limit theorem also gives that the edge conditions become independent in the limit so that $\lim_{n\rightarrow\infty}\Gamma^{{A}}=\overline{\alpha}^{-e}$.  
If $k=1$ then $\overline{\alpha}=(n+1)2^{-n}$ and $[\Gamma^{\overline{A}}]_{K_3}=(3n+1)4^{-n}$ so $[\Gamma^{{A}}]_{K_3}=1-(3n+3)2^{-n}+(3n^2+3n)4^{-n}$.   
  
\bibliographystyle{amsplain}

\def\cprime{$'$} \def\cprime{$'$}
\providecommand{\bysame}{\leavevmode\hbox to3em{\hrulefill}\thinspace}
\providecommand{\MR}{\relax\ifhmode\unskip\space\fi MR }
\providecommand{\MRhref}[2]{%
  \href{http://www.ams.org/mathscinet-getitem?mr=#1}{#2}
}
\providecommand{\href}[2]{#2}

\vskip20pt
\noindent{\bf Authors:} Eric Babson and Matthias Beck.  
\end{document}